\documentclass{gtpart}     
\usepackage{graphicx,color} 
\usepackage{mathtools,amsmath}
 \usepackage{tikz}
\usepackage[utf8]{inputenc}
\usepackage[T1]{fontenc}
\usepackage{amssymb}
\usepackage{graphicx}

\title{Hyperbolic structures for Artin-Tits groups of spherical type}

\author{Matthieu Calvez}
\givenname{Matthieu}
\surname{Calvez}
\address{Matthieu Calvez, 
Departamento de Matem\'{a}tica y Estad\'istica , Universidad de La Frontera, Francisco Salazar 1145, Temuco, Chile}
              \email{matthieu.calvez@ufrontera.cl}

\author{Bert Wiest}
\givenname{Bert}
\surname{Wiest}
\address{Bert Wiest, Univ Rennes, CNRS, IRMAR - UMR 6625, F-35000 Rennes, France}
\email{bertold.wiest@univ-rennes1.fr}


\volumenumber{}
\issuenumber{}
\publicationyear{}
\papernumber{}
\startpage{}
\endpage{}
\doi{}
\MR{}
\Zbl{}
\received{}
\revised{}
\accepted{}
\published{}
\publishedonline{}
\proposed{}
\seconded{}
\corresponding{}
\editor{}
\version{}

\makeautorefname{notation}{Notation}%

\newtheorem{theorem}{Theorem}[section]
\newtheorem{lemma}[theorem]{Lemma} 
\newtheorem{proposition}[theorem]{Proposition}
\newtheorem{corollary}[theorem]{Corollary}
\newtheorem{conjecture}[theorem]{Conjecture} 

\theoremstyle{definition}

\newtheorem{observation}[theorem]{Observation}

\newtheorem{question}[theorem]{Question}

\def\co{\colon \thinspace}
\def\Dn{\mathcal{D}_n} 
\def\Dnpo{\mathcal{D}_{n+1}} 
\def\Bnpo{\mathcal{B}_{n+1}} 

\def\cal{\mathcal{C}_{AL}}  

\renewcommand{\phi}{\varphi}

\def\xdil{1.3} 
\def\eps{.25} 
\def\rad{.1} 

\begin{document}

\begin{abstract}
The goal of this mostly expository paper is to present several candidates for hyperbolic structures on irreducible Artin-Tits groups of spherical type and to elucidate some relations between them. Most constructions are algebraic analogues of previously known hyperbolic structures on Artin braid groups coming from natural actions of these groups on curve graphs and (modified) arc graphs of punctured disks. 
\end{abstract}

\maketitle


\section{Introduction}\label{S:Introduction}
Given a group $G$ and a generating set $X$ of $G$, the word metric $d_X$ turns $G$ into a metric space; this space is $(1,1)$-quasi-isometric to $\Gamma(G,X)$, the Cayley graph of $G$ with respect to $X$, endowed with the usual graph metric where each edge is identified to an interval of length 1. 
Hyperbolic structures on groups were recently introduced and studied in \cite{ABO}. 
A \emph{hyperbolic structure} on a group $G$ is a generating set $X$ of $G$ such that $(G,d_X)$ is Gromov-hyperbolic; note that $X$ must be infinite whenever $G$ is not itself hyperbolic.

In this paper, we are interested in hyperbolic structures on Artin-Tits groups of spherical type. One motivation is trying to prove that irreducible Artin-Tits groups of spherical type 
are hierarchically hyperbolic \cite{BehrstockHagenSisto2, BehrstockHagenSisto3}, where the hierarchical structure should come from the hierarchy of parabolic subgroups of these groups.

After \cite{CalvezWiest1,CalvezWiest2}, we know that irreducible Artin-Tits groups of spherical type admit non-trivial hyperbolic structures; i.e.\ each irreducible Artin-Tits group of spherical type~$A$ contains an (infinite) generating set $X_{abs}^A$ such that the corresponding Cayley graph $\Gamma(A,X_{abs}^A)$  is a Gromov-hyperbolic metric space with \emph{infinite diameter}. This hyperbolic structure was defined in a purely algebraic manner, using only the Garside structure on~$A$, and it consists of the set of the so-called \emph{absorbable elements} (to which we must add the cyclic subgroup generated by the square of the so-called Garside element if $A$ is of dihedral type). We shall briefly review this construction in Section \ref{S:ATHyp}.

Unfortunately, these absorbable elements are poorly understood -- for instance we do not know any polynomial-time algorithm which recognizes whether any given element belongs to $X_{abs}^A$ -- and this makes it quite difficult to work with the graph $\Gamma(A,X_{abs}^A)$.  

In this paper we generalize to any irreducible Artin-Tits group of spherical type some well-known hyperbolic structures on Artin's braid groups with $n+1$ strands $\mathcal B_{n+1}$ (a.k.a.\ Artin-Tits groups of type~$A_n$), $n\geqslant 3$. 

Because it can be identified with the mapping class group of a $n+1$ times punctured disk $\mathcal D_{n+1}$, Artin's braid group on $n+1$ strands admits nice actions on the curve graph of $\Dnpo$ (denoted by $\mathcal C(\Dnpo)$), the arc graph of $\Dnpo$ (denoted by $\mathcal A(\Dnpo)$) and the graph of arcs in $\Dnpo$ both of whose extremities lie in $\partial\Dnpo$ (denoted by $\mathcal A_{\partial}(\Dnpo)$). All these graphs can be shown to be connected and Gromov-hyperbolic; this was first shown in \cite{MasurMinsky1} but the circle of ideas around \cite{HPW,PS} provides simpler arguments. 

All these actions are cobounded (actually cocompact); 
according to a standard argument (Lemma \ref{L:Main}, close in spirit to Svarc-Milnor's lemma  \cite[Section 3.2]{ABO}), 
we extract from each of these actions 
a hyperbolic structure on $\mathcal B_{n+1}$, which consists of the union of the stabilizers of a (finite) family of representatives of the orbits of vertices. 

Each of these generating sets can be algebraically described in terms of the \emph{parabolic subgroups} of the Artin-Tits group of type $A_n$, allowing to extend the definitions to any irreducible Artin-Tits group of spherical type.
Given an irreducible Artin-Tits group of spherical type $A$, we define:
\begin{itemize}
\item $X_P^A$ is the union of all proper irreducible standard parabolic subgroups of $A$ and the cyclic subgroup generated by the square of the Garside element; 
\item $X_{NP}^A$ is the union of the normalizers of all proper irreducible standard parabolic subgroups of $A$;
\item $X_{abs}^A$ is the set of absorbable elements (together with the cyclic subgroup generated by the square of the Garside element, if $A$ is of dihedral type: 2 generators).
\end{itemize}

Note that $X_{NP}^A$ contains the cyclic subgroup generated by the square of the Garside element (which is central). Similarly for $X_{abs}^A$: any power of the Garside element can be written as a product of at most 3 absorbable elements, provided $A$ is not of dihedral type  \cite[Example 3]{CalvezWiest1}. Therefore the center of $A$ has bounded diameter with respect to the word metric on $A$ induced by any of the above generating sets. 

We then study the relationships between $X_P^A$, $X_{NP}^A$ and  $X_{abs}^A$. Following~\cite{ABO}, given two generating sets $X,Y$ of a group $G$, we write $X\preccurlyeq Y$ if the identity map from $(G,d_Y)$ to $(G,d_X)$ is Lipschitz (or equivalently, if $\sup_{y\in Y} d_X(1_G,y)<\infty$). The sets $X$ and $Y$ are \emph{equivalent} if both $X\preccurlyeq Y$ and $Y\preccurlyeq X$ hold (or equivalently, if the identity map is a bilipschitz equivalence between $(G,d_X)$ and $(G,d_Y)$).

Table \ref{Table} summarizes the main contents of this paper, for any irreducible Artin-Tits group of spherical type $A$ with at least 3 generators. Vertical arrows indicate the identity of $A$. The \emph{graph of irreducible parabolic subgroups} and the \emph{additional length graph} (denoted by $\mathcal C_{parab}(A)$ and $\mathcal C_{AL}(A)$, respectively) were defined in \cite{CGGMW} and \cite{CalvezWiest1}, respectively. 
For Artin-Tits groups \emph{of type $A$}, all the mentioned generating sets are hyperbolic structures. In any case, all spaces under consideration have infinite diameter (Corollary \ref{C:InfDiam}). 

\newcommand{\xdownarrow}[1]{%
  {\left\downarrow\vbox to #1{}\right.\kern-\nulldelimiterspace}
}
\begin{center}
\begin{table}[h]
\begin{tabular}{|c|ccc|c|}
\hline
\multicolumn{1}{|c|}{Gen.\ set} & \multicolumn{3}{c|}{Cayley graph} & Hyperbolicity\\
\hline
$X_P^A$ & & $\Gamma(A,X_P^A)$ & generalized $\mathcal A_{\partial}(\Dnpo)$ & Conjectured\\

\rotatebox[origin=c]{270}{$\preccurlyeq$}  & & $\xdownarrow{1cm}$ & Lipschitz. Conj.\,\ref{C:StrictInequalities}(ii): not equivalent& \\

$X_{NP}^A$ & & $\Gamma(A,X_{NP}^A)$ & generalized $\mathcal C(\Dnpo)$, & Conjectured\\
& & & q.isom. to $\mathcal C_{parab}$ \cite{CGGMW} & \\
\rotatebox[origin=c]{270}{$\preccurlyeq$}  & & $\xdownarrow{1cm}$ & Lipschitz \cite{CalvezWiest1,AntolinCumplido}. Conj.\,\ref{C:StrictInequalities}(i): equivalent& \\
  
$X_{abs}^A$ & & $\Gamma(A,X_{abs}^A)$ &  q.isom. to $\mathcal C_{AL}(A)$ \cite{CalvezWiest1} & Proved \cite{CalvezWiest1,CalvezWiest2}\\
\hline
\end{tabular}
\caption{Summary of the main results presented in the paper when $A$ has at least 3 generators.}\label{Table}
\end{table}
\end{center}

A space similar to $\Gamma(A,X_P^A)$ was defined in \cite{CharneyCrisp}; that paper gives necessary and sufficient conditions on an Artin-Tits group (not necessarily of spherical type) for the union of its standard parabolic subgroups of spherical type to be a hyperbolic structure. However, when the group is itself of spherical type, this results in a Cayley graph of finite diameter. 
It is also interesting to note that a construction analogous to $\mathcal C_{parab}(A)$ has been recently announced --and conjectured to be hyperbolic-- for each Artin-Tits group of FC type \cite{morris}.  

When $A$ is of spherical dihedral type, all three sets $X_P^A$, $X_{NP}^A$ and $X_{abs}^A$ are hyperbolic structures but their orders differ slightly, namely $X_P^A$ and $X_{NP}^A$ are equivalent to each other and smaller than $X_{abs}^A$ -- see Proposition~\ref{P:Dihedral}; also, all corresponding Cayley graphs have infinite diameter. 


The plan of the paper is as follows. In Section \ref{S:Prelim} we review necessary definitions and facts about Artin-Tits groups of spherical type; we also prove Lemma \ref{L:Main}. The latter is used in Section~\ref{S:BnHyp} to describe hyperbolic structures on $\mathcal B_{n+1}$; at the same time, these structures are reinterpreted algebraically in terms of parabolic subgroups of $\mathcal B_{n+1}$. In Section~\ref{S:ATHyp}, we justify the assertions and present the conjectures from Table~\ref{Table}.


\section{Preliminaries}\label{S:Prelim}

\subsection{Preliminaries on Artin-Tits groups of spherical type}\label{SS:PrelimArtin}

Let $S$ be a finite set. A \emph{Coxeter matrix} over $S$ is a symmetric matrix $M=(m_{s,t})_{s,t\in S}$ with $m_{s,s}=1$ for $s\in S$ and $m_{s,t}\in \mathbb N_{\geqslant 2}\cup \{\infty\}$ for $s\neq t\in S$.  The \emph{Coxeter graph} 
$\Gamma=\Gamma_S=\Gamma(M)$ associated to $M$ is a labeled graph whose vertices are the elements of~$S$ and in which two distinct vertices are connected by an edge labeled by $m_{s,t}$ whenever $m_{s,t}\geqslant 3$.

For two symbols $a,b$ and an integer $m$, define 
$$\Pi(a,b;m)=\begin{cases} (ab)^k & \text{if $m=2k$,}\\ (ab)^ka & \text{if $m=2k+1$.}\\ \end{cases}$$
The Artin-Tits system associated to $M$ (or to $\Gamma$) is the pair $(A,S)$ where $A=A_S=A_{\Gamma}$ is the group presented by 
$$\left\langle S\right.\left|\ \Pi(s,t;m_{s,t})=\Pi(t,s;m_{s,t}), \forall s\neq t\in S, \text{with}\ m_{s,t}<\infty\right\rangle.$$

As relations are homogeneous, there is a well-defined homomorphism $\ell:A_S\longrightarrow \mathbb Z$ which sends each $s\in S$ to 1. The \emph{rank} of $A_S$ is the number of elements of $S$. 

Given a subset $T$ of $S$, denote by $\Gamma_T$ the subgraph of $\Gamma_S$ induced by the vertices in~$T$ and by $A_T$ the subgroup of $A_S$ generated by $T$. It is a theorem of Van der Lek \cite{VanDerLek} 
that $(A_T,T)$ is the Artin-Tits system associated to the graph $\Gamma_T$. Such a subgroup $A_T$ ($T\subset S$) is called a \emph{standard parabolic subgroup} of $A_S$. A \emph{parabolic subgroup} of $A_S$ is a subgroup which is conjugate to some standard parabolic subgroup. 

The group $A_S$ is said to be \emph{irreducible} if the graph $\Gamma_S$ is connected, equivalently if $A_S$ cannot be expressed as the direct product of two standard parabolic subgroups. The group $A_S$ is of \emph{spherical type} if the Coxeter group obtained as the
quotient of $A_S$ by the normal closure of the squares of elements in $S$ is finite. The group $A_S$ is said to be of \emph{dihedral type} if it has rank~2 (in this case, the corresponding Coxeter group is exactly the usual dihedral group).

Irreducible Artin-Tits groups of spherical type have been classified \cite{Coxeter} (see also \cite{Humphreys} for a modern exposition) into type $A_n$~($n\geqslant 1$), $B_n$~($n\geqslant 2$), $D_n$~($n\geqslant 4$), $E_6$, $E_7$, $E_8$, $F_4$,  $H_3$, $H_4$ and $I_{2m}$ ($m\geqslant 5$). Figure \ref{F:Coxeter} summarizes this classification. Note that for $n\geqslant 1$, the Artin-Tits group of type $A_n$ is precisely Artin's braid group $\mathcal B_{n+1}$. In this paper, we will be concerned only with irreducible Artin-Tits groups of spherical type. For an integer $m=3,4$, we write also $I_{2m}$ for $A_2$ and $B_2$ respectively. 

\begin{figure}
\includegraphics[width=13cm]{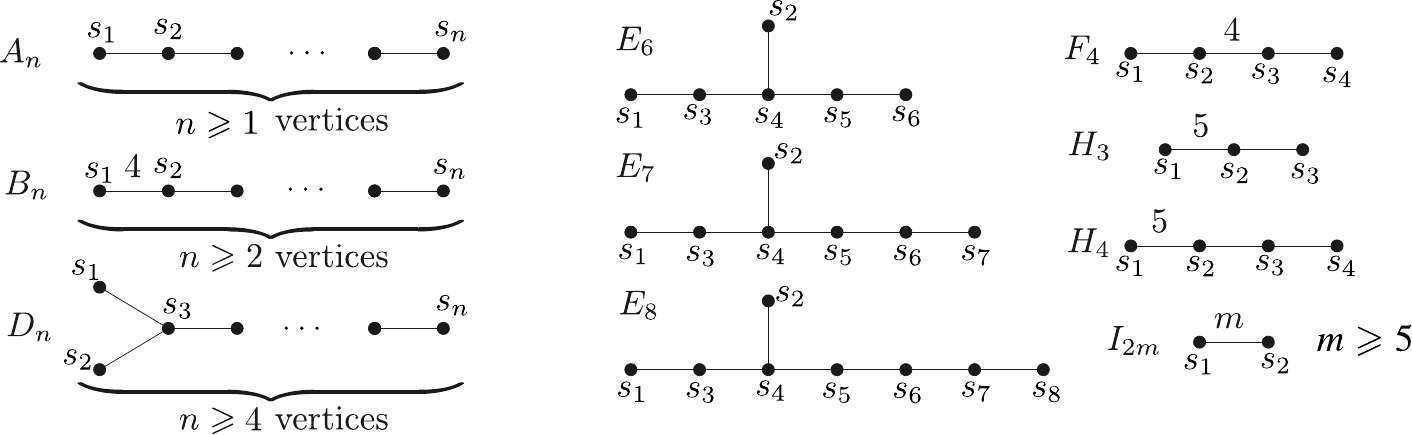}
\caption{Coxeter's classification of irreducible Artin-Tits groups of spherical type}\label{F:Coxeter}
\end{figure}

We recall that each irreducible Artin-Tits groups of spherical type $A=A_S$ can be equipped with a finite Garside structure -- this was proved in \cite{BrieskornSaito,Deligne}.
The reader is referred to \cite[Section 2]{GebhardtGM} for a gentle introduction to Garside structures. We just recall the existence of a Garside element $\Delta_S$; either $\Delta_S$ or $\Delta_S^2$ generates the center of~$A_S$ and we denote by $\Omega_S$ this minimal central element. 
(It is known that $\Omega_S=\Delta_S$ for groups of type $A_1$, $B_n$ ($n\geqslant 2$), $D_n$~($n$~even), $E_7$, $E_8$,  $F_4$, $H_3$, $H_4$ and $I_{2m}$~($m$~even), and $\Omega_S=\Delta_S^2$ for groups of type $A_n$ ($n\geqslant 2)$, $D_n$~($n$~odd), $E_6$, and $I_{2m}$~($m$~odd)). Similarly, for each irreducible standard parabolic subgroup $A_T$ ($T\subset S$), there is a Garside element $\Delta_T$ and $\Omega_T=\Delta_T$ or $\Delta_T^2$ generates the center of $A_T$. Moreover, if $P$ is any irreducible parabolic subgroup of $A_S$, then $P=a^{-1}A_T a$ for some element~$a\in A$ and some $T\subset S$ -- note that this expression is by no means unique. However there is a well-defined element $\Omega_P=a^{-1}\Omega_Ta$ which does not depend on $a$ nor on $T$ in the above expression of $P$ and generates the center of $P$ (see \cite{CGGMW}). 

The two following results on conjugacy of parabolic subgroups will be useful later.

\begin{proposition}\cite[Theorem 6.1]{Paris}\label{P:Paris}
Let $A=A_S$ be an irreducible Artin-Tits group of spherical type and let $T\subset S$ such that $A_T$ is an irreducible Artin-Tits group. Then 
$$N_{A_S}(A_T)=Z_{A_S}(\Omega_T).$$\end{proposition}

\begin{proposition}\cite[Section 11]{CGGMW}\label{P:SimultStandard}
Let $A=A_S$ be an irreducible Artin-Tits group of spherical type. Let $P,Q$ be two irreducible parabolic subgroups of $A_S$ and suppose that $\Omega_P$ and $\Omega_Q$ commute. Then there exists $g\in A$ and $T_1,T_2\subset S$ such that 
$g^{-1}Pg=A_{T_1}$ and $g^{-1}Qg=A_{T_2}$.
\end{proposition}

We conclude this section with a few words on the case of dihedral type Artin-Tits groups.

\begin{proposition}\label{P:PureFree}\cite[Theorem 6.2 (2)] {McCammond}
Let $m\geqslant 3$ be an integer. Let $A$ be an Artin-Tits group of type $I_{2m}$. Then the quotient $A/\negthinspace\left\langle \Delta^2\right\rangle$ has a finite index subgroup isomorphic to a free group of rank $m-1$.
\end{proposition}

\begin{corollary}\label{C:Dihedral}
Let $m$ and $A$ be as in Proposition \ref{P:PureFree}. Let $X$ be any finite generating set of $A$. 
Then the generating set $X\cup\langle \Delta^2\rangle$ is a hyperbolic structure on $A$. All hyperbolic structures obtained in this way are equivalent.
\end{corollary}

\begin{proof}
Let $\overline X$ be the set of cosets of elements of $X$ mod $\langle \Delta^2\rangle$; this generates $A/\langle\Delta^2\rangle$.  Observe that $\Gamma(A,X\cup\langle\Delta^2\rangle)$ and $\Gamma(A/\langle \Delta^2\rangle, \overline X)$ are quasi-isometric spaces. From Proposition \ref{P:PureFree} it follows applying Svarc-Milnor lemma that $\Gamma(A/\langle \Delta^2\rangle, \overline X)$ is quasi-isometric 
to a $m-1$-regular tree, hence hyperbolic. Therefore, for any $X$, $\Gamma(A,X\cup\langle \Delta^2\rangle)$ is quasi-isometric to a regular tree, whence the claim. 
\end{proof}


\subsection{Cocompact actions on connected graphs}

The unique result in this subsection will allow us to exhibit a number of hyperbolic structures on Artin's braid group $\mathcal B_{n+1}$. Although this is a standard fact, we include a stand-alone proof for the convenience of the reader.

\begin{lemma}\cite[Lemma 3.2]{MasurMinsky1}\cite[Lemma 10]{KimKoberda}\label{L:Main}
Let $X$ be a connected graph and denote by $d_X$ the graph-metric on $X$ obtained by identifying each edge to an interval of length~1. 
Let $G$ be a group acting by isometries on $X$ such that 
\begin{itemize}
\item[(i)] there exists a \emph{finite} set $A$ of vertices of $X$ such that every vertex orbit has a (non-necessarily unique) representative in $A$,
\item[(ii)] there exists a \emph{finite} set $B$ of edges of $X$ such that every edge orbit has a (non-necessarily unique) representative in $B$,
\item[(iii)] the set $T=\bigcup_{a\in A} Stab_G(a)$ generates $G$. 
\end{itemize}
Then $(G,d_T)$ is quasi-isometric to $(X,d_X)$. 
\end{lemma}

\begin{proof}
It is known that $(X,d_X)$ is (1,1)-quasi-isometric to $(\mathcal V_X,d_X)$, the vertex set of $X$ endowed with the restriction of the graph metric on $X$; we focus on $(\mathcal V_X,d_X)$ rather than $X$ itself.

Let $M_1= Diam (A)$. Let $A_0\subset A$ having exactly one representative of each vertex orbit under the action of $G$; for each $a\in A$, choose $g_a\in G$ such that $g_a\cdot a\in A_0$ and let $M_2=\max\{\|g_a\|_T, \ a\in A\}$ (where we denote $\|g_a\|_T=d_T(1_G,g_a)$). 

Let $\{\{u_i,v_i\},1\leqslant i\leqslant m\}$ be an enumeration of the finite set $B$ and 
fix $\alpha_i,\beta_i\in G$ so that $\alpha_i^{-1}\cdot u_i\in A$ and $\beta_i^{-1}\cdot v_i\in A$. 
Let $$M_3=\max_{1\leqslant i\leqslant m}\{\|\alpha_i\|_T, \|\beta_i\|_T\}.$$

Fix $a_0\in A$ and define a map $\psi\co G\longrightarrow \mathcal V_X$ by $g\mapsto g\cdot a_0$.
We first prove that the map $\psi\co (G,d_T)\to (\mathcal V_X,d_X)$ is Lipschitz.
Let $g,h\in G$, let $n=\|g^{-1}h\|_T$ and write $g^{-1}h=t_1\ldots t_n$, with $t_i\in T$. We have 
$$d_X(\psi(g),\psi(h))=d_X(a_0,g^{-1}h\cdot a_0)=d_X(a_0,t_1\ldots t_n\cdot a_0)$$
and
\begin{eqnarray*}
 d_X(a_0,t_1\ldots t_n\cdot a_0) & \leqslant  & d_X(a_0,t_1\cdot a_0)+d_X(t_1\cdot a_0,t_1t_2\cdot a_0)+\ldots\\
    & & \quad \quad\ldots + d_X(t_1\ldots t_{n-1}\cdot a_0,t_1\ldots t_n\cdot a_0)\\
    & = & \sum_{i=1}^n d_X(a_0,t_i\cdot a_0).
    \end{eqnarray*}
But for any $t\in T$ there is at least some $a_t\in A$ so that $t\cdot a_t=a_t$; so we have 
$$d_X(a_0,t\cdot a_0)\leqslant d_X(a_0,a_t)+d_X(a_t,t\cdot a_0)=d_X(a_0,a_t)+d_X(t\cdot a_t,t\cdot a_0)=2d_X(a_0,a_t)\leqslant 2M_1.$$

In summary, we have $$d_X(\psi(g),\psi(h))\leqslant 2M_1d_T(g,h).$$

On the other hand, given any vertex $x$ of $X$, there exist $a\in A$ and  $g\in G$ so that $x=g\cdot a$. If $a,a'\in A$ and $g,g'\in G$ are such that $g\cdot a=g'\cdot a'=x$, we have $g_a\cdot a=g_{a'}\cdot a'\in A_0$ whence $g'^{-1}gg_a^{-1}g_a'$ fixes $a'$ and so lies in $T$. It follows that the set of elements $g\in G$ such that $g\cdot a=x$ for some $a\in A$ 
has diameter at most $1+2M_2$ in $(G,d_T)$. 
We define $\varphi:\mathcal V_X\longrightarrow G$ in the following way: to each $x\in \mathcal V_X$ we associate some $g$ in $G$ such that $g^{-1}\cdot x\in A$. It is now clear that $ \phi\circ\psi$ is at distance at most $2M_2+1$ from $Id_G$  and $\psi\circ\phi$ is at distance at most~$M_1$ from $Id_{\mathcal V_X}$.

We now show that $\phi\co(\mathcal V_X,d_X)\longrightarrow (G,d_T) $ is also Lipschitz. Let $x_1,x_2$ be adjacent vertices of $X$. There is some $g$ in $G$ and $1\leqslant i\leqslant m$ such that $\{g\cdot x_1,g\cdot x_2\}=\{u_i,v_i\}$
is an edge belonging to~$B$.
By construction, $\alpha_i^{-1}g\cdot x_1=\alpha_i^{-1}\cdot u_i\in A$ and $\beta_i^{-1}g\cdot x_2=\beta_i^{-1}\cdot v_i\in A$; it follows that 
$d_T(\phi(x_1),g^{-1}\alpha_i)\leqslant 2M_2+1$ and $d_T(\phi(x_2),g^{-1}\beta_i)\leqslant 2M_2+1$. 

We then have 
\begin{eqnarray*}
d_T(\phi(x_1),\phi(x_2)) & \leqslant & d_t(\phi(x_1),g^{-1}\alpha_i)+d_T(g^{-1}\alpha_i,g^{-1}\beta_i)+d_T(g^{-1}\beta_i,\phi(x_2))\\
 & \leqslant & 4M_2+2+d_T(\alpha_i,\beta_i)\\
  & \leqslant &  4M_2+2+2M_3.
 \end{eqnarray*}
 
In summary, if $x_1, x_2$ are arbitrary vertices of~$X$ then 
$$d_T(\phi(x_1),\phi(x_2))\leqslant 2(2M_2+M_3+1)\cdot d_X(x_1,x_2).\qedhere$$ 
\end{proof}


\section{Hyperbolic structures for $\mathcal B_{n+1}$}\label{S:BnHyp}

In this section, we review some Gromov-hyperbolic graphs on which Artin's braid group $\Bnpo$ acts. Lemma \ref{L:Main} is put to work in order to derive hyperbolic structures on $\Bnpo$ from these actions. These hyperbolic structures are interpreted in an algebraic context. 

\subsection{The curve graph}\label{Subsection:CurveGraph}

Fix an integer  $n\geqslant 3$. Artin's braid group on $n+1$ strands is identified with the mapping class group of a closed, $n+1$ times punctured disk $\Dnpo$ (with boundary fixed pointwise).
We consider the following model for $\Dnpo$: the disk of diameter $[0,n+2]$ in $\mathbb C$ with the standard orientation,
and $n+1$ punctures at $1,\ldots, n+1$ numbered accordingly $p_1,\ldots, p_{n+1}$. 
For $1\leqslant i\leqslant n$, write $s_i$ for the half-Dehn twist along an horizontal arc connecting $p_i$ and $p_{i+1}$. Then $(\Bnpo,\{s_1,\ldots, s_n\})$ is exactly the Artin-Tits system associated to the Coxeter graph~$A_n$. 

The \emph{curve graph} $\mathcal C(\Dnpo)$ of the $n+1$ times punctured disk $\Dnpo$ is the (locally-infinite) graph defined as follows: 
\begin{itemize} 
\item its vertices are isotopy classes of essential simple closed curves in $\Dnpo$, that is isotopy classes of simple closed curves in $\Dnpo$  surrounding at least 2 and at most $n$ punctures;
\item two distinct vertices are connected by an edge if and only if the corresponding isotopy classes can be realized disjointly.
\end{itemize}

According to \cite{MasurMinsky1}, this graph is connected and Gromov-hyperbolic; a simpler proof can be found in \cite{HPW}. 
The group $\mathcal B_{n+1}$ acts naturally by isometries (simplicial automorphisms) on $\mathcal C(\Dnpo)$.

An essential simple closed curve in $\Dnpo$ (a vertex of $\mathcal C(\Dnpo)$) is said to be \emph{round} or \emph{standard} if it is isotopic to a geometric circle. Note that there are exactly $\frac{n(n+1)}{2}-1$ such vertices: 
each of them can be described by two integers $1\leqslant i<j\leqslant n+1$, where $p_i$ ($p_j$ respectively) is the first (the last respectively) puncture surrounded; we denote this curve by $c_{ij}$. 
For $1\leqslant i<j\leqslant n+1$, we denote by $\mathcal B_{ij}$ the set of braids whose support is enclosed by $c_{ij}$; this is the standard irreducible parabolic subgroup generated by $\{s_t, i\leqslant t\leqslant j-1\}$, which is isomorphic to a braid group on $j-i+1$ strands. The square of the Garside element of this group $\Delta_{ij}^2$ is the Dehn twist around the curve $c_{ij}$.

\begin{lemma}\label{L:StructureXC}
Let $X_{\mathcal C}$ be the union of stabilizers of all standard curves. 
Then $\Bnpo$ equipped with the word metric $d_{X_{\mathcal C}}$ is quasi-isometric to the curve graph of the $n+1$ times punctured disk.
\end{lemma}

\begin{proof} 
First observe that the action of each generator $s_i$ fixes some standard curve; therefore, $X_{\mathcal C}$ generates $\mathcal B_{n+1}$. 
According to \cite[Section 1.3.1]{FarbMargalit}, the set $A$ 
of all standard curves contains a representative of each vertex orbit under the action of $\Bnpo$.
According to \cite[Proposition 4.4]{LeeLee}, the \emph{finite} set of edges in $\mathcal C (\Dnpo)$ both of whose extremities are standard curves contains a representative of each edge orbit under the action of $\Bnpo$. Hence we see that the isometric action of $\Bnpo$ on the curve graph $\mathcal C(\Dnpo)$ satisfies the hypotheses of Lemma~\ref{L:Main}.
\end{proof}

\begin{proposition}\label{P:Curves}
Let $X_{NP}^{\Bnpo}$ be the union of all normalizers of  proper irreducible standard parabolic subgroups of $\mathcal B_{n+1}$. 
Then $\Bnpo$ equipped with the word metric $d_{X_{NP}^{\Bnpo}}$ is quasi-isometric to the curve graph of the $n+1$ times punctured disk. 
In particular, $X_{NP}^{\Bnpo}$ is a hyperbolic structure on $\Bnpo$.
\end{proposition}

\begin{proof} According to \cite[Fact 3.8]{FarbMargalit}, the stabilizer of any curve $c$ is exactly the centralizer of the Dehn twist around $c$; 
if $c=c_{ij}$ is standard, this is the centralizer of the square of the Garside element $\Delta_{ij}$ of $\mathcal B_{ij}$ (this is the same as the centralizer of $s_i$ if $j=i+1$). In other words, in view of Proposition~\ref{P:Paris}, $Stab_{\Bnpo}(c_{ij})$ is the same as the normalizer of the standard parabolic subgroup~$\mathcal B_{ij}$. 
\end{proof}

%
%
%
%

\subsection{The boundary connecting arc graph} 

Again, we fix an integer $n\geqslant 3$; we keep the same model for the punctured disk $\Dnpo$ and $s_i$, $1\leqslant i\leqslant n$, are the usual Artin generators of $\Bnpo$. 

The \emph{boundary connecting arc graph} $\mathcal A_{\partial}(\Dnpo)$ of the $n+1$ times punctured disk $\Dnpo$ is the (locally-infinite) graph defined as follows: 

\begin{itemize} 
\item its vertices are isotopy classes of essential (i.e. not homotopic into a subset of $\partial \Dnpo$) simple arcs in $\mathcal D_{n+1}$ both of whose endpoints lie in the boundary $\partial \mathcal D_{n+1}$ and where in addition the endpoints are allowed to slide freely along the boundary during the isotopy.
\item two distinct vertices are connected by an edge if and only if the corresponding isotopy classes of arcs can be realized disjointly. 
\end{itemize}

The argument given in \cite{HPW} can be easily adapted to show that $\mathcal A_{\partial}(\Dnpo)$ is a Gromov-hyperbolic connected graph. The group $\Bnpo$ acts by isometries on $\mathcal A_{\partial}(\Dnpo)$. 

Given $1\leqslant i\leqslant n$, let $a_i$ be the arc consisting simply of the vertical line with real part $i+\frac{1}{2}$; let 
$A=\{a_i, 1\leqslant i \leqslant n\}$.
Then $A$ exhausts all
the orbits of vertices of $\mathcal A_{\partial}(\Dnpo)$ under the action of $\Bnpo$ (the orbit of a boundary connecting arc~$a$ is uniquely characterized by the number of punctures in any of the two connected components of $\Dnpo\setminus a$). Two different disjoint boundary connecting arcs (forming an edge of $\mathcal A_{\partial}(\Dnpo)$) can be transformed simultaneously into two arcs $a_i,a_j$ for some $1\leqslant i\neq j\leqslant n$ by the action of a braid: this says that there is a finite system $B$ of representatives of the orbits of edges. Finally, each half-twist $s_i$ stabilizes some arc in~$A$. (Indeed, $s_i$ fixes $a_{i-1}$ if $i\geqslant 2$, and $s_1$ fixes $a_2$.)
Therefore, all three hypotheses of Lemma \ref{L:Main} are satisfied. 

It follows from Lemma \ref{L:Main} that the union $X_{\mathcal A_{\partial}}$ of stabilizers of arcs $a_i, \ 1\leqslant i \leqslant n$, is a hyperbolic structure on $\Bnpo$ and that $(\mathcal B_{n+1}, d_{X_{\mathcal A_{\partial}}})$ is quasi-isometric to $\mathcal A_{\partial}(\Dnpo)$.


{\bf {Claim.}} 
If $n$ is odd and $i=\frac{n+1}{2}$, the stabilizer of $a_i$ is generated by $\Delta$ and the direct product of standard parabolic subgroups $\langle s_1,\ldots,s_{i-1}\rangle\times \langle s_{i+1}\ldots s_n\rangle$. 
In any other case, the stabilizer of $a_i$ is generated by $\Delta^2$ and the direct product of standard parabolic subgroups $\langle s_1,\ldots,s_{i-1}\rangle\times \langle s_{i+1}\ldots s_n\rangle$.

\begin{proof}
The arc $a_i$ cuts $\mathcal D_{n+1}$ into two connected components which are punctured disks with $i$ and $n+1-i$ punctures respectively. Except if $2i=n+1$, an element $g$ in the stabilizer of $a_i$ must preserve the boundary of each of these two disks. Following the conventions in \cite{LeeLee}, this means that $g$ can be written as a product of an interior braid (which is an element of $\langle s_1,\ldots,s_{i-1}\rangle\times \langle s_{i+1}\ldots s_n\rangle$) and a tubular braid of the form 
$$\left((s_i\ldots s_1)(s_{i+1}\ldots s_2)\ldots (s_n\ldots s_{n-i+1})(s_{n-i+1}\ldots s_1)\ldots (s_n\ldots s_{i+1})\right)^k$$
 for some $k\in \mathbb Z$ (a pure braid on 2 fat strands).  We observe however that the above braid can be written as 
 $$\Delta^{2k}\Delta^{-2k}_{\langle s_1,\ldots,s_{i-1}\rangle}\Delta^{-2k}_{\langle s_{i+1}\ldots s_{n}\rangle}.$$
 If $2i=n+1$, an element of the stabilizer of $a_i$ can be written as the product of an interior braid as above and a tubular braid (on 2 fat strands) which needs not be pure, that is a power of 
 $$(s_i\ldots s_1)(s_{i+1}\ldots s_2)\ldots (s_n\ldots s_{n-i+1}),$$
which we can rewrite as 
$$\Delta\Delta^{-1}_{\langle s_1,\ldots,s_{i-1}\rangle}\Delta^{-1}_{\langle s_{i+1}\ldots s_{n}\rangle}.\qedhere$$
\end{proof}


\begin{proposition}\label{P:PartialA} Let $X_{P}^{\Bnpo}$ be the subset of $\Bnpo$ 
consisting of the union of all proper irreducible standard parabolic subgroups together with the center $\langle \Delta^2\rangle$.  Then $\Bnpo$ equipped with the word metric $d_{X_{P}^{\Bnpo}}$ is quasi-isometric to $\mathcal A_{\partial}(\Dnpo)$. In particular, $X_{P}^{\Bnpo}$ is a hyperbolic structure on $\Bnpo$.
\end{proposition}

\begin{proof}
Let $X_{\mathcal A_{\partial}}$ be the union of stabilizers of all arcs $a_i$, $1\leqslant i\leqslant n$, as above. Then $X_{\mathcal A_{\partial}}$ contains the generating set $X_{P}^{\Bnpo}$ (every element in a proper irreducible standard parabolic subgroup fixes some arc $a_i$).
On the other hand, each element in $X_{\mathcal A_{\partial}}$ can be written as a product of at most 6 elements of $X_P^{\Bnpo}$, according to the above claim (observe that $\Delta$ can be decomposed as a product of three elements of proper irreducible standard parabolic subgroups). 
This shows that both generating sets are equivalent. The remaining part of the statement follows as $X_{\mathcal A_{\partial}}$ is a hyperbolic structure and $(\mathcal B_{n+1}, d_{X_{\mathcal A_{\partial}}})$ is quasi-isometric to $\mathcal A_{\partial}(\Dnpo)$.
\end{proof}


\section{Generalizing to Artin-Tits groups of spherical type}\label{S:ATHyp}

Throughout this section, $A=A_S$ is some fixed non-cyclic irreducible Artin-Tits group of spherical type. 

\subsection{Definitions}

\subsubsection{Generalized boundary connecting arc graph}

Let $X_P^A$ be the union of all proper irreducible standard parabolic subgroups of $A$ and the cyclic subgroup generated by the square of $\Delta_S$. It is clear that $X_P^A$ generates $A$ because each generator in $S$ belongs to some proper irreducible standard parabolic subgroup. 

\begin{proposition} Suppose that one of the following holds. 
\begin{itemize}
\item[(i)] $A$ is of dihedral type, or
\item[(ii)] $A$ is of type $A_n$ ($n\geqslant 3$). 
\end{itemize}
Then $X_P^A$ is a hyperbolic structure on $A$. 
\end{proposition}
\begin{proof}
We refer to Proposition \ref{P:Dihedral} for the proof of~(i). For (ii), this is Proposition \ref{P:PartialA}: in this case, the Cayley graph $\Gamma(A, X_P^A)$ is quasi-isometric to the graph $\mathcal A_{\partial}(\mathcal D_{n+1})$. 
\end{proof}

\begin{conjecture}
Let $A$ be any non-cyclic irreducible Artin-Tits group of spherical type. The set $ X_P^A$ is a hyperbolic structure on $A$. 
\end{conjecture}

\subsubsection{Generalized curve graph}

Let $X_{NP}^A$ be the union of the normalizers of all proper irreducible standard parabolic subgroups of $A$. 
Of course, $X_{NP}^A$ generates $A$, as each $s\in S$ normalizes any parabolic subgroup containing it and hence belongs to $X_{NP}^A$. Note also that any power of $\Delta_S^2$ normalizes every parabolic subgroup. 

Recall from \cite{CGGMW} the definition of the \emph{graph of parabolic subgroups} of $A$, denoted $\mathcal C_{parab}(A)$:
\begin{itemize}
\item its vertices are proper irreducible parabolic subgroups of $A$,
\item two distinct vertices $P,Q$ are connected by an edge if 
the minimal central elements $\Omega_P$ and $\Omega_Q$ commute.
\end{itemize}

The group $A$ acts naturally (by conjugation on vertices) by isometries on $\mathcal C_{parab}(A)$. Observing that the set of \emph{standard} parabolic subgroups is finite and using Proposition~\ref{P:SimultStandard}, Lemma~\ref{L:Main}
implies along the same lines as 
Lemma~\ref{L:StructureXC}:

\begin{proposition}
 The graphs $\Gamma(A,X_{NP}^A)$ and $\mathcal C_{parab}$ are quasi-isometric. 
\end{proposition}

In the special case where $A_S$ is of type $A_n$ ($n\geqslant 3$), it is shown in \cite{CGGMW} that $\mathcal C_{parab}(A_S)$ is exactly the same graph (isometric) as the curve graph of a $n+1$ times punctured disk. 
We have: 
\begin{proposition} Suppose that one of the following holds. 
\begin{itemize}
\item[(i)] $A$ is of dihedral type,  or
\item[(ii)] $A$ is of type $A_n$ ($n\geqslant 3$). 
\end{itemize}
Then $X_{NP}^A$ is a hyperbolic structure on $A$. 
\end{proposition}
\begin{proof}
We refer to Proposition \ref{P:Dihedral} for the proof of~(i). For (ii), this is Proposition \ref{P:Curves}: in this case, the Cayley graph $\Gamma(A, X_{NP}^A)$ is quasi-isometric to the graph $\mathcal C(\mathcal D_{n+1})$. 
\end{proof}

\begin{conjecture}
Let $A$ be any non-cyclic irreducible Artin-Tits group of spherical type. The set $X_{NP}^A$ is a hyperbolic structure on $A$.
\end{conjecture}

\subsubsection{Absorbable elements and additional length graph}

Let $X_{abs}^A$ be the set of absorbable elements (together with the cyclic subgroup generated by $\Delta^2$ if $A$ is of dihedral type). The definition of absorbable element of a Garside group was proposed in \cite{CalvezWiest1}. It is not difficult to see that each element in $S$ is absorbable, hence $X_{abs}$ generates~$A$. 

Recall from \cite{CalvezWiest1} the definition of the \emph{additional length graph} $\mathcal C_{AL}(A)$: its vertices are right cosets of $A$ modulo the cyclic subgroup generated by $\Delta$; there is an edge between the vertices $a\langle \Delta\rangle$ and $b\langle\Delta\rangle$ if there is some $u$ which is either simple (i.e. a positive prefix of $\Delta$) or absorbable such that $au$ belongs to the coset $b\langle\Delta\rangle$. 

\begin{lemma}\label{P:CALQI}
 $\mathcal C_{AL}(A)$ and $\Gamma(A,X_{abs}^A)$ are quasi-isometric. 
 \end{lemma}

\begin{proof}
It is enough to check the quasi-isometry of the respective vertex sets. Any coset $a\langle\Delta\rangle$ has diameter at most 4 in $\Gamma(A,X_{abs}^A)$: 
any two elements of such a coset differ by a power of~$\Delta$; and any power of~$\Delta$ can be written as a product of at most 4 elements of $X_{abs}^A$  in the dihedral case (because $\Delta^q=\Delta^{2\lfloor \frac{q}{2}\rfloor} \cdot a \cdot b \cdot b^{-1}a^{-1}\Delta$), and of at most 3 absorbable elements in the general case: (the argument of 
\cite[Example 3]{CalvezWiest1} for braid groups with at least 4 strands can be easily adapted to any Artin-Tits group of rank $\geqslant 3$).
For any coset $a\langle\Delta\rangle$, choose an arbitrary representative; this defines a map~$\phi$ from the vertices of $\mathcal C_{AL}$ to $A$. Conversely, given $a\in A$, we associate to~$a$ the coset $\psi(a)=a\langle\Delta\rangle$.  
It is straightforward to check that these maps are quasi-inverses 
of each other with respect to the graph metrics on $\mathcal C_{AL}(A)$ and $\Gamma(A,X_{abs}^A)$. It is also clear that $\psi$ is 1-Lipschitz. On the other hand, as the set of simple elements (positive divisors of $\Delta$) is finite, we can define $M$ as the maximum of $d_{X_{abs}^A}(1,s\Delta^j)$, for $s$ simple and $j\in \mathbb Z$. It then follows that $\phi$ is $M$-Lipschitz.
\end{proof}

The graph $\mathcal C_{AL}(A)$ was shown to be hyperbolic in \cite{CalvezWiest1}; therefore Lemma \ref{P:CALQI} shows:

\begin{proposition}
Let $A$ be any non-cyclic irreducible Artin-Tits group of spherical type. The set $X_{abs}^A$ is a hyperbolic structure on $A$. 
\end{proposition}

\subsection{Comparisons and further results}

In this subsection, we compare the different generating sets for $A$ described above. 
We will first look at Artin-Tits groups of dihedral type, and then at all other types. 

\begin{proposition}\label{P:Dihedral}
Assume that $A$ is of type $I_{2m}$ ($m\geqslant 3$). Then we have 
$$X_P^A \sim X_{NP}^A\preccurlyeq X_{abs}^A\sim S\cup\langle\Delta^2\rangle,$$
where the middle inequality is strict. All the sets mentioned are hyperbolic structures on $A$ and the corresponding Cayley graphs all have infinite diameter.
\end{proposition}

\begin{proof}
Denote by $a,b$ the generators of $A$. 
In order to prove that $X_P^A \sim X_{NP}^A$, we first 
observe that $X_P^A\subset X_{NP}^A$ whence $X_{NP}^A\preccurlyeq X_P^A$. For the converse inequality, it will be enough to show that the normalizer of $\langle a\rangle$ ($\langle b\rangle$ respectively) is generated by $\Omega$ and $a$ ($\Omega$ and $b$ respectively). 
In order to do so, we first observe that for any $g\in N_A(\langle a\rangle)$ we have $g^{-1} a g\in \langle a\rangle$, and using the abelianisation we see that $g^{-1} a g=a$, i.e., 
the normalizer of $\langle a\rangle$ coincides with $Z_A(a)$, the centraliser of $a$. This centraliser can be very easily computed using the algorithm in \cite{FrancoGM}. For $m$ even, we get that $Z_A(a)=\langle a, \Delta\rangle$. 
For $m$ odd, we obtain $\{a, \Pi(b,a;m-2)a^2\Pi(b,a;m-2)\}$ as a generating set; up to right and left multiplying the second generator by $a$, this gives the generating set $\{a,\Delta^2\}$. In any case, an element of $N_A(\langle a \rangle)$ can be written as $\Omega^p a^q$, which has length at most 
$(2+\|\Delta\|_{X_P^A})$ with respect to $d_{X_P^A}$ -- and similarly for $N_A(\langle b\rangle)$.
It follows that the identity map from $(A,d_{X_{NP}^A})$ to $(A,d_{X_{P}^A})$ is  $(2+\|\Delta\|_{X_P^A})$-Lipschitz.
This shows the first equivalence. 

In order to prove the middle inequality, we note that the set of absorbable elements is finite with $4m-8$ elements (see \cite[Example 2.2]{CalvezWiest1} for $m=3$).
Letting $M=\sup\{\|x\|_{X_P ^A},\ x\in X_{abs}^A\setminus\langle \Delta^2\rangle\}$, we see that the identity is an $M$-Lipschitz map from $(A,d_{X_{abs}^A})$ to $(A, d_{X_P^A})$, whence the middle inequality.
This inequality is not an equivalence, because we have $d_{X_P^A}(Id, a^n)=1$ for all~$n$ but $d_{X_{abs}^A}(Id, a^n)\to_{n\to \infty} +\infty$.

For the second equivalence, we use again that there are only finitely many absorbable elements. It follows from Corollary \ref{C:Dihedral} that 
$X_{abs}^A$ and $S\cup\langle\Delta^2\rangle$ are equivalent hyperbolic structures on $A$. 


Next we prove that $X_P^A$ is a hyperbolic structure. The proof uses the ``guessing geodesics lemma''  \cite[Theorem 3.15]{MS}.


Given $g,h$ vertices of $X_P^A$, consider the subgraph $A_{gh}$ consisting of the projections to $\Gamma(A,X_P^A)$ of the set of geodesics in $\Gamma(A,S\cup \langle \Delta^2 \rangle)$ between $g$ and $h$. Because the identity from $\Gamma(A,S\cup \langle \Delta^2 \rangle)$ to $\Gamma(A,X_P^A)$  is Lipschitz and $\Gamma(A,S\cup \langle \Delta^2 \rangle)$ is hyperbolic, the different subgraphs $A_{g,h}$ form thin triangles in $\Gamma(A,X_P^A)$ so they satisfy the second condition of \cite[Theorem 3.15]{MS}.

What happens if $g$, $h$ are 1 apart in $\Gamma(A,X_P^A)$? 
Observe first that if $g,h\in A$ satisfy $g^{-1}h=a^n$ (or $b^n$, with $n\in \mathbb Z$), then in $\Gamma(A,S\cup \langle \Delta^2 \rangle)$, there is a unique geodesic between $g$ and $h$, namely $g, ga, ga^2,\ldots, h$. 
Thus the subgraph $A_{gh}$ has diameter 1 as well; hence the first condition in \cite[Theorem 3.15]{MS} is also satisfied.

We conclude that $\Gamma(A,X_P^A)$ is hyperbolic.

Finally, we have to prove that $\Gamma(A,X_P^A)$ has infinite diameter.
Let us assume, for a contradiction, that there exists some number~$L$ such that every element $x$ of~$A$ can be written in the form 
$$
x=\Delta^{k_0} a^{k_1} b^{k_2} a^{k_3} b^{k_4} \ldots c^{k_L} \text{ \ \ or \ \ }
x=\Delta^{k_0} b^{k_1} a^{k_2} b^{k_3} a^{k_4} \ldots c^{k_L} 
$$
with $k_0, k_1, \ldots, k_L\in\mathbb Z$, and $c=a$ or $c=b$, according to the parity of~$L$. Now, we recall from Proposition \ref{P:PureFree} that $\Gamma(A,S\cup \langle \Delta^2 \rangle)$ is quasi-isometric to an infinite regular $2(m-1)$-valent tree -- in particular, as a metric space it has uncountably many ends. Also, powers of~$a$ and powers of~$b$ represent quasi-geodesics in this quasi-tree.  However, the set of points in the tree represented by words as above has only countably many ends, leading to a contradiction.
%
\end{proof}

For the rest of the section, we now assume that $A$ has rank at least 3. 

\begin{proposition}\label{P:Comparation3}
Suppose that $A$ is an irreducible Artin-Tits group of spherical type with rank at least 3. 
Then
$X_{abs}^A\preccurlyeq X_{NP}^A\preccurlyeq X_P^A$. 
\end{proposition}

\begin{proof}
For the second inequality, it is enough to observe that $X_P^A\subset X_{NP}^A$ whence the identity map $(A,d_{X_P^A})\longrightarrow (A,d_{X_{NP}^A})$ is 1-Lipschitz.

The first inequality was shown for braids in \cite{CalvezWiest1}: any braid in the stabilizer of a standard curve can be written as a product of 9 absorbable braids \cite[Lemma 11]{CalvezWiest1}. This result has been generalized to the case of a general Artin-Tits group $A$ of spherical type in the forthcoming paper \cite{AntolinCumplido}, where it is shown that any element which normalizes a proper irreducible standard parabolic subgroup of $A$ is a product of at most 9 absorbable elements. This is to say that the identity map from $(A,d_{X_{NP}^A})$ to $(A,d_{X_{abs}^A})$ is 9-Lipschitz. 
\end{proof}

\begin{corollary}\label{C:InfDiam}
The metric spaces $(A,d_{X_{NP}^A})$ and $(A,d_{X_P^A})$ have infinite diameter. 
\end{corollary}
\begin{proof}
This is just a combination of Proposition \ref{P:Comparation3} and Lemma \ref{P:CALQI} together with \cite[Theorem 1.1]{CalvezWiest2}, which asserts that $\mathcal C_{AL}(A)$ has infinite diameter.
\end{proof}

\subsection{Open problems}

\begin{conjecture}\label{C:StrictInequalities}
\begin{itemize}
\item[(i)] The inequality $X_{abs}^A\preccurlyeq X_{NP}^A$ is not strict:  the converse inequality $X_{NP}^A\preccurlyeq X_{abs}^A$ also holds, and the identity map 
$\Gamma(A,X_{NP}^A) \to \Gamma(A,X_{abs}^A)$ is a quasi-isometry. 
(In the case of Artin braid group, this is claiming that $\mathcal C_{AL}(\mathcal{B}_n)$ is quasi-isometric to the curve graph of the punctured disk $\mathcal C(\Dn)$.)   
\item[(ii)] The inequality $X_{NP}^A\preccurlyeq X_P^A$ is strict, i.e., 
the identity map $\Gamma(A,X_{P}^A)\to \Gamma(A,X_{NP}^A)$ is Lipschitz but \emph{not} a quasi-isometry.
\end{itemize}
\end{conjecture} 


The truth of Conjecture~\ref{C:StrictInequalities}(i) would of course imply that $X_{NP}^A$ is a hyperbolic structure (and hence hyperbolicity of the graph of irreducible parabolic subgroups~$\mathcal C_{parab}(A)$). 
It would also imply that Garside normal forms are unparameterized quasi-geodesics in~$\mathcal C_{parab}(A)$ (since they are in $\mathcal C_{AL}(A)$); this would contrast with the recently announced example in \cite{RafiVerberne} of a family of geodesics in a mapping class group whose shadows in the corresponding curve graph are not unparameterized quasi-geodesics. 

Note that Conjecture~\ref{C:StrictInequalities}(i) is not even known to hold in the specific case of Artin's braid groups -- see \cite[Conjecture 1]{CalvezWiest1}. Indeed, as pointed out to us by Ursula Hamenstädt, it is very much conceivable that $\mathcal C_{AL}(\mathcal{B}_n)$ is quasi-isometric to a tree. 

\begin{proposition} Conjecture~\ref{C:StrictInequalities}(ii) does hold when $A$ is a braid group with at least 4 strands: for $n\geqslant 3$, the natural map $\mathcal A_{\partial}(\Dnpo) \to \mathcal C(\Dnpo)$ is not a quasi-isometry. 
\end{proposition}
\begin{proof} 
We refer to \cite[Section 5]{MS}. Consider the open disk $D$ of radius 1 centered at~$\frac{3}{2}$ containing the first two punctures of $\mathcal D_{n+1}$ (see the beginning of Section \ref{Subsection:CurveGraph}); let $X =\mathcal D_{n+1}\setminus D$: 
this subsurface is a hole for $\mathcal A_{\partial}(\mathcal D_{n+1})$ according to \cite[Definition 5.2]{MS}. Note also that $X$ is homeomorphic to $\mathcal D_n$. 
Now, let $\beta$ be a pseudo-Anosov braid on $n$ strands and consider the $n+1$-strand braid $\hat\beta$ obtained from $\beta$ by doubling the first strand. Fix $c$ greater than the constant $C_0$ relative to the complex $\mathcal A_{\partial}(\mathcal D_{n+1})$ in \cite[Theorem 5.14]{MS}.
It is known \cite[Proposition 4.6]{MasurMinsky1} that $\beta$ acts in a loxodromic way on the curve graph of $\mathcal D_n$, hence it acts loxodromically on $\mathcal C(X)$, the curve graph of $X$. In other words, we can find $\alpha>0$ 
so that 
$$d_{\mathcal C(X)}(1,\beta^k)\geqslant \alpha k,$$ for every integer $k$, and in particular $d_{\mathcal C(X)}(1,\beta^k)\geqslant c,$ for $k$ big enough. Now, \cite[Theorem 5.14]{MS} implies that there exists a constant $A=A(c)\geqslant 1$ such that for $k$ big enough, $d_{\mathcal A_{\partial}(\mathcal D_{n+1})}(1,{\hat\beta}^k)\geqslant \frac{\alpha k-A}{A}$, 
whence $\hat \beta$ acts loxodromically on $\mathcal A_{\partial}(\mathcal D_{n+1})$. However, as it preserves the round curve bounding~$D$, $\hat\beta$ acts elliptically on $\mathcal C(\mathcal D_{n+1})$.  This finishes the proof of Conjecture~\ref{C:StrictInequalities}(ii) for braid groups.
\end{proof}

\begin{question}
Is every large piece of 2-dimensional quasi-flat in the Cayley graph of a mapping class group sent to a subset of bounded diameter by the projection to the curve complex?
\end{question}

This question is deliberately vague, but to the best of our knowledge, no result along these lines is known, except that \emph{maximal-dimensional} quasi-flats are known to be squashed down to bounded diameter in the curve complex~\cite{BehrstockMinsky,BKMM}. We will  give a more precise version of the question below, and we will show that a positive answer would imply Conjecture~\ref{C:StrictInequalities}(i).

Let us look at the Cayley graph of~$A$ with respect to the Garside generators, modulo the $\Delta$-action: for any $z\in A$, all the vertices corresponding to elements of the form $x\Delta^k$ ($k\in\mathbb Z$) get identified, and if any two edges in the Cayley graph have the same endpoints after the $\Delta$-action, then they get identified as well. The quotient space is quasi-isometric to the Cayley graph of $A/Z(A)$, and we will denote it $Cay(A)/\langle\Delta\rangle$ -- see \cite{CalvezWiest1} for a detailed account. 

Now suppose an element $y$ is absorbed by~$x$, which by \cite[Lemma 3]{CalvezWiest1} can be supposed to have the same length as~$y$; thus $\inf(x)=\inf(y)=\inf(xy)=0$ and $\sup(x)=\sup(y)=\sup(xy)=L$. We will study the implications for the geometry of the graph $Cay(A)/\langle\Delta\rangle$. There is an equilateral triangle with corners $1$, $x$ and~$xy$ (red in Figure~\ref{F:AbsTriang}), and with sides of length~$L$ (blue in the figure) representing the Garside normal form words for $x$, $y$, and $xy$. 
Let us denote the Garside normal form words by $x=x_1 x_2 \ldots x_L$, and $y=y_1 y_2\ldots y_L$. Now, if $1\leqslant \ell \leqslant L$, then $x$ also absorbs $y_1\ldots y_\ell$. This means that in $Cay(A)/\langle\Delta\rangle$ we have $d(1,xy_1\ldots y_\ell)=L$: every vertex on the edge between $x$ and~$xy$ is at the same distance (namely~$L$) from the opposite corner of the triangle (namely~$1$). More generally, if $1\leqslant \ell^y \leqslant \ell^x\leqslant L$, then $x_{L-\ell^x+1}\ldots x_L$ absorbs $y_1\ldots y_{\ell^y}$, and $d(x_1\ldots x_{L-\ell^x}, xy_1\ldots y_{\ell^y})=\ell^x$.

\begin{figure}[htb] 
\begin{center}
\pgfdeclarelayer{triangle}
\pgfdeclarelayer{vertices}
\pgfsetlayers{main,triangle,vertices} 
\colorlet{trcolor}{blue}
\colorlet{vercolor}{red}
\begin{tikzpicture}[scale=0.9]
\draw (0,-\eps) -- (0,5+\eps);
\draw (\xdil*.5,-\eps) -- (\xdil*.5,5+\eps);
\draw (\xdil,-\eps) -- (\xdil,5+\eps);
\draw (\xdil*1.5,-\eps) -- (\xdil*1.5,5+\eps);
\draw (\xdil*2,-\eps) -- (\xdil*2,5+\eps);
\draw (\xdil*2.5,-\eps) -- (\xdil*2.5,5+\eps);
\draw (-\xdil*\eps,-\eps) -- (\xdil*2.5+\xdil*\eps,2.5+\eps);
\draw (-\xdil*\eps,1-\eps) -- (\xdil*2.5+\xdil*\eps,3.5+\eps);
\draw (-\xdil*\eps,2-\eps) -- (\xdil*2.5+\xdil*\eps,4.5+\eps);
\draw (-\xdil*\eps,3-\eps) -- (\xdil*2+\xdil*\eps,5+\eps);
\draw (-\xdil*\eps,4-\eps) -- (\xdil*1+\xdil*\eps,5+\eps);
\draw (-\xdil*\eps,5-\eps) -- (\xdil*\eps,5+\eps);
\draw (\xdil*1-\xdil*\eps,-\eps) -- (\xdil*2.5+\xdil*\eps,1.5+\eps);
\draw (\xdil*2-\xdil*\eps,-\eps) -- (\xdil*2.5+\xdil*\eps,.5+\eps);
\draw (-\xdil*\eps,5+\eps) -- (\xdil*2.5+\xdil*\eps,2.5-\eps);
\draw (-\xdil*\eps,4+\eps) -- (\xdil*2.5+\xdil*\eps,1.5-\eps);
\draw (-\xdil*\eps,3+\eps) -- (\xdil*2.5+\xdil*\eps,.5-\eps);
\draw (-\xdil*\eps,2+\eps) -- (\xdil*2+\xdil*\eps,-\eps);
\draw (-\xdil*\eps,1+\eps) -- (\xdil*1+\xdil*\eps,-\eps);
\draw (-\xdil*\eps,+\eps) -- (\xdil*\eps,-\eps);
\draw (\xdil*1-\xdil*\eps,5+\eps) -- (\xdil*2.5+\xdil*\eps,3.5-\eps);
\draw (\xdil*2-\xdil*\eps,5+\eps) -- (\xdil*2.5+\xdil*\eps,4.5-\eps);

\begin{pgfonlayer}{triangle}
\draw[trcolor,line width=1.5] (0,0) -- (0,5) -- (\xdil*2.5,2.5) -- (0,0);
\filldraw [trcolor] (0,1) circle [radius=\rad];
\filldraw [trcolor] (0,2) circle [radius=\rad];
\filldraw [trcolor] (0,3) circle [radius=\rad];
\filldraw [trcolor] (0,4) circle [radius=\rad];
\filldraw [trcolor] (\xdil*.5,.5) circle [radius=\rad];
\filldraw [trcolor] (\xdil*1,1) circle [radius=\rad];
\filldraw [trcolor] (\xdil*1.5,1.5) circle [radius=\rad];
\filldraw [trcolor] (\xdil*2,2) circle [radius=\rad];
\filldraw [trcolor] (\xdil*.5,4.5) circle [radius=\rad];
\filldraw [trcolor] (\xdil*1,4) circle [radius=\rad];
\filldraw [trcolor] (\xdil*1.5,3.5) circle [radius=\rad];
\filldraw [trcolor] (\xdil*2,3) circle [radius=\rad];
\end{pgfonlayer}

\begin{pgfonlayer}{vertices}
\filldraw [vercolor] (0,0) circle [radius=\rad];
\filldraw [vercolor] (0,5) circle [radius=\rad];
\filldraw [vercolor] (\xdil*2.5,2.5) circle [radius=\rad];

\draw[vercolor] (-0.4,0) node{$1$};
\draw[vercolor] (-0.5,5) node{$xy$};
\draw[vercolor] (\xdil*2.5+0.4,2.5) node{$x$};
\end{pgfonlayer}
\end{tikzpicture}
\end{center}
\caption{A typical fat equilateral triangle in $Cay(A)/\langle\Delta\rangle$, here with $L=5$. For instance, in $A=\mathcal B_4$, one can take $x=\sigma_1^5$, $y=\sigma_3^5$ and $xy=(\sigma_1\sigma_3)^5$.}
\label{F:AbsTriang}
\end{figure}
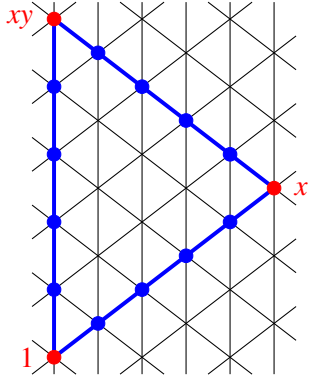

Another crucial observation is that the three sides of the triangle play completely symmetrical roles: if $x$ absorbs~$y$, then $y$ absorbs $(xy)^{-1}\Delta^L$ and $\Delta^L (xy)^{-1}$ absorbs $x$. 
In particular, every vertex on every side of the triangle (between 1 and~$x$, or between $x$ and~$xy$, or between $xy$ and~$1$) is at distance~$L$ from the opposite corner. More generally, if two vertices of $Cay(A)/\langle\Delta\rangle$ lie on two different sides of the triangle, and if they are at distance $d_1$ and $d_2$, respectively, from the corner of the triangle shared by the two sides, then their distance in~$Cay(A)/\langle\Delta\rangle$ is $\max(d_1,d_2)$. Let us call such a triangle ``fat'' -- indeed, such a triangle is not at all thin (in the sense of $\delta$-hyperbolicity), but on the contrary all distances are at least as large as in a Euclidean comparison triangle.

\begin{question}\label{Q:FatTriangleQuestion}
We have seen that every absorbable element gives rise to a \emph{fat} equilateral triangle in $Cay(A)/\langle\Delta\rangle$. Are the images of all such triangles in $\Gamma(A,X_{NP}^A)$ (or, equivalently, in the complex of irreducible parabolic subgroups) of uniformly bounded diameter? For $A=\Bnpo$, do the images in $\mathcal C(\Dnpo)$ of all fat equilateral triangles have uniformly bounded diameter?
\end{question}

Notice that for every fat triangle induced from an absorbable element, the image in $\mathcal C_{AL}(A)$ is of diameter at most~$2$ (because the image of every edge is of diameter~$1$). 

\begin{observation} A positive answer to Question~\ref{Q:FatTriangleQuestion}  would imply that Conjecture~\ref{C:StrictInequalities}(i)  is true.
\end{observation}


\end{document}